\documentclass[12pt,a4paper]{article}
\usepackage[latin5]{inputenc}
\usepackage[T1]{fontenc}
\usepackage{amsmath}
\numberwithin{equation}{section}
\usepackage{amsfonts}
\usepackage{amssymb}
\usepackage{amsthm}
\usepackage{graphicx}
\usepackage{enumerate}
\usepackage{enumitem}
\usepackage{multicol}
\usepackage[all]{xy}
\usepackage{xcolor}
\usepackage{soul}
\usepackage{faktor}
\usepackage{xfrac}
\usepackage{bm}
\usepackage[bookmarks=true,
bookmarksnumbered=true, breaklinks=true,
pdfstartview=FitH, hyperfigures=false,
plainpages=false, naturalnames=true,
colorlinks=true, pdfpagelabels]{hyperref}
\usepackage{geometry}
\usepackage{palatino}
\allowdisplaybreaks
\usepackage{indentfirst}
\usepackage{tocloft}
\usepackage{authblk}
\usepackage{setspace}
\theoremstyle{definition}
\newtheorem{theorem}{Theorem}[section]
\newtheorem{proposition}[theorem]{Proposition}
\newtheorem{lemma}[theorem]{Lemma}
\newtheorem{corollary}[theorem]{Corollary}

\newtheorem{definition}[theorem]{Definition}
\newtheorem{example}[theorem]{Example}

\def\Gwa{\mathbf{Gr}^{\bullet}}
\def\rGwa{\mathbf{rGr}^{\bullet}}
\def\asc{{\ast}^{\circ}}

\title{\LARGE \textbf{Actions and semi-direct products in categories of groups with action}}

\author[a]{Tamar Datuashvili\thanks{\textbf{Corresponding author: }\texttt{tamar.datu@gmail.com} (T. Datuashvili)}}
\author[b]{Tunçar Şahan}
\affil[a]{\small{A. Razmadze Mathematical Institute of I. Javakhishvili Tbilisi State University, 6 Tamarashvii Str., Tbilisi 0177, Georgia}}
\affil[b]{\small{Department of Mathematics, Aksaray University, Aksaray, Turkey}}
\date{}

\begin{document}
\maketitle

\begin{abstract}
Derived actions in the category of groups with action on itself $\mathbf{Gr}^{\bullet}$ are defined and described. This category plays a crucial role in the solution of Loday's two problems stated in the literature. A full subcategory of reduced groups with action $\mathbf{rGr}^{\bullet}$ of $\mathbf{Gr}^{\bullet}$ is introduced, which is not a category of interest but has some properties, which can be applied in the investigation of action representability in this category; these properties are similar to those, which were used in the construction of universal strict general actors in the category of interest. Semi-direct product constructions are given in $\mathbf{Gr}^{\bullet}$ and $\mathbf{rGr}^{\bullet}$ and it is proved that an action is a derived action in $\mathbf{Gr}^{\bullet}$ (resp. $\mathbf{rGr}^{\bullet}$) if and only if the corresponding semi-direct product is and object of $\mathbf{Gr}^{\bullet}$ (resp. $\mathbf{rGr}^{\bullet}$). The results obtained in this paper will be applied in the forthcoming paper on the representability of actions in the category $\mathbf{rGr}^{\bullet}$.\\[0.4cm]
\end{abstract}

\section{Introduction}

We develop action theory in the category of groups with action on itself $\Gwa$, introduced in \cite{Datuashvili2002a,Datuashvili2004,Datuashvili2017}, where it played a main role in the solution of Loday's two problems stated in \cite{Loday1993,Loday2003}. This category is neither a category of interest in the sense of Orzech \cite{Orzech1972,Orzech1972a}, nor a modified category of interest \cite{Boyaci2015}. It is a category of groups with operations, but doesn't satisfy all conditions stated in \cite{Porter1987}. The category $\Gwa$ is a category of $\Omega$-groups in the sense of Kurosh \cite{Kurosh1965}. Actions are defined in $\Gwa$ as derived actions from split extensions in this category as it is in the category of interest or in any semi-abelian category \cite{Borceux2005}. We describe derived action conditions in this category and construct a semi-direct product $B\ltimes A$, where $A,B\in\Gwa$ and $B$ has a derived action on $A$. We prove that an action of $B$ on $A$ is a derived action if and only if $B\ltimes A\in\Gwa$ (Theorem \ref{theo:der-act-equiv}). Then we define a full subcategory $\rGwa $ in $\Gwa$ and describe derived actions in $\rGwa$. Our interest is to investigate the existence of a universal acting object on an object $G\in\Gwa$ applying the results obtained in \cite{Casas2007,Casas2010} for categories of interest. Since the category $\Gwa$ is far from being category of interest we found its  subcategory $\rGwa$, which is not a category of interest, but has interesting properties which are close to those ones we use in the construction of a universal strict general actor for any object of a category of interest in \cite{Casas2007,Casas2010}. We prove necessary and sufficient condition for the action of $B$ on $A$, $A, B\in\rGwa$, to be a derived action in terms of the semi-direct product $B\ltimes A$, like we have in $\Gwa$ (Theorem \ref{theo:equivcond}). Applying the results of this paper, we will prove that under certain conditions on the object $G\in\rGwa$, it has representable actions in the sense of \cite{Borceux2005}, i.e. a universal acting object, which represents all actions on $A$.

\section{Preliminary definitions and results}

Let $G$ be a group which acts on itself from the right side, i.e. we have a map $\varepsilon\colon G\times G\rightarrow G$ with
\begin{alignat*}{2}
\varepsilon(g, g'+g'') & = \varepsilon(\varepsilon(g, g'), g'') \\
\varepsilon(g, 0) & = g \\
\varepsilon(g'+g'', g) & = \varepsilon(g'+g) + \varepsilon(g''+g)
\end{alignat*}
for $g, g', g''\in G$. Denote $\varepsilon(g, h) =g^{h}$, for $g, h\in G$.

We denote the group operation additively, nevertheless the group is not commutative in general. From the third condition on $\varepsilon$ it follows that
\begin{equation*}
	0^h = 0, \text{ for any } h\in G.
\end{equation*}

If $(G', \varepsilon')$ is another group with action then a homomorphism $(G, \varepsilon)\rightarrow(G', \varepsilon')$ is a group homomorphism $\varphi\colon G\rightarrow G'$, for which the diagram
\[
\xymatrix{
	G\times G \ar[r]^-{\varepsilon} \ar[d]_{\varphi\times\varphi} & G \ar[d]^{\varphi} \\
	G'\times G' \ar[r]_-{\varepsilon'} & G'}
\]
commutes. In other words, we have
\begin{equation}\label{cond:gwamorph}
	\varphi\left(g^h\right) = {\varphi(g)}^{\varphi(h)}
\end{equation}
for all $g, h\in G$.

Note that action defined above is a split derived action in the sense of \cite{Orzech1972, Orzech1972a}.

According to Kurosh \cite{Kurosh1965} an $\Omega$-group is a group with a system of $n$-ary algebraic operations $\Omega_{n\geq 0}$, which satisfy the condition
\begin{equation}\label{cond:omgrp}
000\cdots 0\omega = 0,
\end{equation}
where 0 is the identity element of $G$, and 0 on the left side occurs $n$ times if $\omega$ is an $n$-ary operation. In special cases $\Omega$-groups give groups, rings, associative and non-associative algebras like Lie and Leibniz algebras etc. and groups with action on itself as well. In the latter case $\Omega$ consists of one binary operation which is an action or $\Omega$ consists of only unary operations, which are elements of $G$, and this operation is an action again. In both cases condition \eqref{cond:omgrp} is satisfied. Denote the category of groups with action on itself by $\Gwa$; here the action is considered as a binary operation and morphisms between the objects in $\Gwa$ are group homomorphisms satisfying condition \eqref{cond:gwamorph}.

Let $G\in\Gwa$.

\begin{definition}\cite{Datuashvili2002a}
A non-empty subset $A$ of $G$ is called an ideal of $G$ if it satisfies the following conditions
\begin{enumerate}[label={\textbf{(\arabic{*})}}, leftmargin=1cm]
	\item\label{cond:defideal1} $A$ is a normal subgroup of $G$ as a group;
	\item\label{cond:defideal2} $a^{g}\in A$, for any $a\in A$ and $g\in G$;
	\item\label{cond:defideal3} $-g+g^a\in A$, for any $a\in A$ and $g\in G$.
\end{enumerate}
\end{definition}

Note that the condition \ref{cond:defideal3} in this definition is equivalent to the condition, that $g^a-g\in A$, since $(-g)^a=-g^a,$ for any $a\in A$ and $g\in G$. This definition is equivalent to the definition of an ideal given in \cite{Kurosh1965} for $\Omega$-groups in the case where $\Omega$ consists of one binary operation of action, one can see the proof in \cite{Datuashvili2002a}.

\section{Actions and semi-direct products in $\Gwa$}

Let $A, B\in\Gwa$. An action of $B$ on $A$ by definition is a triple of mappings $\beta=(\beta_{+}, \beta_{\ast}, \beta_{\asc})\colon B\times A\rightarrow A$, where $\ast$ is a binary operation of action, $\asc$ is its dual operation in $\Gwa$, i.e. $\beta_{+}(b,a)= b\cdot a$, $\beta_{\ast}(b,a)=a\ast b= {a}^{b}$ and $\beta_{\asc}(b,a)=a\asc b={b}^{a}$.

In the category of interest or category of groups with operations there is a condition $0\ast g=g\ast0=0$, for any binary operation $\ast\in\Omega\backslash\{+\}$, any object $G$ in this category and any element $g\in G$. In the category $\Gwa$ we have ${0}^{g}=0$, for any $G\in\Gwa$ and any $g\in G$, but ${g}^{0}\neq 0$ in general. Therefore we modify the definition of derived action due to split extensions \cite{Orzech1972,Orzech1972a,Porter1987}, known for the category of groups with operations or category of interest, for the category $\Gwa$. Note that the definition of derived action from the split extension agrees with the definition of action in a semi-abelian category \cite{Borceux2005}.

Let $A,B\in\Gwa$. An extension of $B$ by $A$ is a sequence
\begin{equation}\label{ses}
\xymatrix{0 \ar[r] &   A \ar@{->}[r]^-{i} &   E \ar@{->}[r]^-{p} &   B \ar[r] & 0}
\end{equation}
in which $p$ is surjective and $i$ is the kernel of $p$. We say that an extension is split if there is a morphism $j\colon B\rightarrow E$, such that $pj=1_B$. We will identify $i(a)$ with $a$.

A split extension induces a triple of actions of $B$ on $A$ corresponding to the operation of addition, action and its dual operation in $\Gwa$. From the split extension \eqref{ses} for any $b\in B$ and $a\in A$ we define
\begin{alignat}{2}
b\cdot a &= j(b)+a-j(b) \label{32}\\
{b}^{a} &={j(b)}^{a}-j(b) \label{33}\\
{a}^{b} &={a}^{j(b)} \label{34}
\end{alignat}

Actions defined by \eqref{32}-\eqref{34} will be called derived actions of $B$ on $A$ as it is in the case of groups with operations or category of interest. Note that \eqref{33} differs from what we have in the noted known cases, since as we have mentioned above ${b}^{0}\neq 0$ in $B$.

\begin{proposition}\label{prop:actcond}
Let $A,B,\in\Gwa$. Derived actions of $B$ on $A$ satisfy the following conditions:
\begin{enumerate}[label={\textbf{(\alph{*})}}, leftmargin=1cm]
\item\label{prop:act-a} well-known group action conditions for the dot left action:
\begin{alignat*}{2}
	b\cdot(a_1+a_2) &= b\cdot a_1 + b\cdot a_2 \\
	(b_1+b_2)\cdot a &= b_1\cdot(b_2\cdot a) \\
	0\cdot a &=a
\end{alignat*}
where $a,a_1,a_2\in A$ and $b,b_1,b_2\in B$,
\item\label{prop:act-b} ${0_{A}}^{b}=0_A$, ${0_B}^{a} = 0_A$, ${b}^{0_{A}}=0_A$, ${a}^{0_{B}}=a$ where $0_A$ and $0_B$ denote the zero elements of $A$ and $B$ respectively. For any $a,a'\in A$ and $b,b'\in B$,
\begin{enumerate}[leftmargin=1.5cm]
	\item[$\bm{(1_{A})}$]\label{prop:act1a} ${(a+a')}^{b}={a}^{b}+{(a')}^{b}$,
	\item[$\bm{(2_{A})}$]\label{prop:act2a} ${(b+b')}^{a}={b}^{a}+b\cdot\left((b')^{a}\right)$,
	\item[$\bm{(3_{A})}$]\label{prop:act3a} ${(b\cdot a)}^{a'}+{b}^{a'}={b}^{a'}+b\cdot\left( {a}^{a'} \right)$,
	\item[$\bm{(4_{A})}$]\label{prop:act4a} $\left(b\cdot a\right)^{b'}={b}^{b'}\cdot {a}^{b'}$,
	\item[$\bm{(1_{B})}$]\label{prop:act1b} ${b}^{(a+a')}=\left({b}^{a}\right)^{a'}+{b}^{a'}$,
	\item[$\bm{(2_{B})}$]\label{prop:act2b} ${a}^{b+b'}=\left({a}^{b}\right)^{b'}$
	\item[$\bm{(3_{B})}$]\label{prop:act3b} $\left({a}^{(b\cdot a')}\right)^{b}=\left({a}^{b}\right)^{a'}$,
	\item[$\bm{(4_{B})}$]\label{prop:act4b} $\left({b}^{(b'\cdot a)}\right)^{b'}=\left({b}^{b'}\right)^{a}$.
\end{enumerate}
\end{enumerate}
\end{proposition}

Note that as it will be shown in the proof of Theorem \ref{theo:der-act-equiv}, all properties noted in \ref{prop:act-b} except ${a}^{0_{B}}=a$ follow from $\bm{(1_A)}$, $\bm{(2_{A})}$ and $\bm{(1_{B})}$. Nevertheless we preferred for explicitness to state in the theorem these properties separately.

\begin{proof}
\begin{enumerate}[label={\textbf{(\alph{*})}}, leftmargin=1cm]
\item This is obvious.
\item Follows from the action properties $({0_E}^{a}=0_E, {a}^{0_{E}}=a, \text{ for all } a\in A)$ and the definition of the derived action corresponding to the action operation and its dual \eqref{33} and \eqref{34}.
\begin{enumerate}[leftmargin=1.5cm]
	\item[$\bm{(1_{A})}$]\label{prf:act1a} Let $a,a'\in A$ and $b\in B$; then
		\begin{alignat*}{2}
		{(a+a')}^{b} &={(a+a')}^{j(b)}\\
		&={a}^{j(b)}+{(a')}^{j(b)}\\
		&={a}^{b}+{(a')}^{b}
		\end{alignat*}
	\item[$\bm{(2_{A})}$]\label{prf:act2a} Let $a\in A$ and $b,b'\in B$; then
		\begin{alignat*}{2}
		{(b+b')}^{a}&=\left(j(b+b')\right)^{a}-j\left(b+b'\right)\\
		&= j\left(b\right)^{a}+j\left(b'\right)^{a}-j\left(b'\right)-j\left(b\right)\\
		&= j\left(b\right)^{a}-j\left(b\right)+j\left(b\right)+j\left(b'\right)^{a}-j\left(b'\right)-j\left(b\right)\\
		&= {b}^{a}+b\cdot\left((b')^{a}\right)
		\end{alignat*}
	\item[$\bm{(3_{A})}$]\label{prf:act3a} Let $a,a'\in A$ and $b\in B$; then
		\begin{alignat*}{2}
		{(b\cdot a)}^{a'}+{b}^{a'}&= \left( j(b)+a-j(b)\right)^{a'}+{j(b)}^{a'}-j(b)\\
		&= j(b)^{a'}+a^{a'}-j(b)^{a'}+{j(b)}^{a'}-j(b)\\
		&= j(b)^{a'}-j(b)+{j(b)}+a^{a'}-j(b)\\
		&= {b}^{a'}+b\cdot\left( {a}^{a'} \right)
		\end{alignat*}
	\item[$\bm{(4_{A})}$]\label{prf:act4a} Let $a\in A$ and $b,b'\in B$; then
		\begin{alignat*}{2}
		\left(b\cdot a\right)^{b'}&= \left(j(b)+a-j(b)\right)^{b'} \\
		&= j(b)^{j(b')}+a^{j(b')}-j(b)^{j(b')}\\
		&= j\left(b^{b'}\right)+a^{j(b')}-j\left(b^{b'}\right)\\
		&= {b}^{b'}\cdot {a}^{b'}
		\end{alignat*}	
	\item[$\bm{(1_{B})}$]\label{prf:act1b} Let $a,a'\in A$ and $b\in B$; then
		\begin{alignat*}{2}
		{b}^{(a+a')}&= {j(b)}^{(a+a')}-j(b) \\
		&= {\left({j(b)}^{a}\right)}^{a'}-{j(b)}^{a'}+{j(b)}^{a'}-j(b) \\
		&= {\left({j(b)}^{a}-{j(b)}\right)}^{a'}+{j(b)}^{a'}-j(b) \\
		&= \left({b}^{a}\right)^{a'}+{b}^{a'}
		\end{alignat*}
	\item[$\bm{(2_{B})}$]\label{prf:act2b} Let $a\in A$ and $b,b'\in B$; then
		\begin{alignat*}{2}
		{a}^{b+b'}&= {a}^{j(b)+j(b')} \\
		&= \left({a}^{b}\right)^{b'}
		\end{alignat*}	
	\item[$\bm{(3_{B})}$]\label{prf:act3b} Let $a,a'\in A$ and $b\in B$; then
		\begin{alignat*}{2}
		\left({a}^{(b\cdot a')}\right)^{b}&= \left({a}^{(j(b)+a'-j(b))}\right)^{j(b)} \\
		&= {a}^{j(b)+a'} \\
		&= \left({a}^{j(b)}\right)^{a'} \\
		&= \left({a}^{b}\right)^{a'}
		\end{alignat*}	
	\item[$\bm{(4_{B})}$]\label{prf:act4b} Let $a,a'\in A$ and $b,b'\in B$; then
		\begin{alignat*}{2}
		\left({b}^{(b'\cdot a)}\right)^{b'}&= \left({j(b)}^{(j(b')+a-j(b'))}-j(b)\right)^{j(b')} \\
		&= \left(\left(\left({j(b)}^{j(b')}\right)^{a}\right)^{-j(b')}\right)^{j(b')}-{j(b)}^{j(b')} \\
		&= \left({j(b)}^{j(b')}\right)^{a}-{j(b)}^{j(b')} \\
		&= \left(j\left({b}^{b'}\right)\right)^{a}-j\left({b}^{b'}\right) \\
		&= \left({b}^{b'}\right)^{a}
		\end{alignat*}
\end{enumerate}
\end{enumerate}
\end{proof}

Given a triple of actions of $B$ on $A$ in $\Gwa$, we can define operations on the product $B\times A$ in the following way:
\begin{alignat}{2}
	(b,a)+(b',a')&=(b+b',a+b\cdot a') \\
	{(b,a)}^{(b',a')}&=({b}^{b'}, ({a}^{a'})^{b'}+ ({b}^{a'})^{b'})
\end{alignat}
for any $(b,a), (b',a')\in B\times A$. This kind of universal algebra will be callaed semi-direct product and denoted by $B\ltimes A$.

\begin{theorem}\label{theo:der-act-equiv}
Let $A,B\in\Gwa$, if $\beta=(\beta_{+}, \beta_{\ast}, \beta_{\asc})$ is a triple of actions of $B$ on $A$, then the following conditions are equivalent:
\begin{enumerate}[label={\textbf{(\arabic{*})}}, leftmargin=1cm]
	\item $\beta$ is a triple of derived actions of $B$ on $A$.
	\item $\beta_{+}$ satisfies group action conditions, $\beta$ satisfies conditions $\bm{(1_A)}-\bm{(4_A)}$, $\bm{(1_B)}-\bm{(4_B)}$ and the condition $a^{0_{B}}=a$, for any $a\in A$.
	\item The semi-direct product $B\ltimes A$ is an object in $\Gwa$.
\end{enumerate}
\end{theorem}

\begin{proof} \begin{enumerate}[leftmargin=2.2cm]
\item[\textbf{(1)}$\Rightarrow$\textbf{(2):}] by Proposition \ref{prop:actcond}.
\item[\textbf{(2)}$\Rightarrow$\textbf{(3):}] First of all we will show that from $\bm{(1_A)}$, $\bm{(2_{A})}$ and $\bm{(1_{B})}$ follow the conditions of \ref{prop:act-b} except the one ${a}^{0_{B}}=a$.
From $\bm{(1_A)}$ we have
\begin{equation*}
	{0_A}^{b}={0_A+0_A}^{b}={0_A}^{b}+{0_A}^{b}
\end{equation*}
and then ${0_A}^{b}=0_A$.

From $\bm{(2_A)}$ we have
\begin{equation*}
	{0_B}^{a}={0_B+0_B}^{a}={0_B}^{a}+{0_B}\cdot\left({0_B}^{b}\right)={0_B}^{a}+{0_B}^{a}
\end{equation*}
and then ${0_B}^{a}=0_A$. Note that we will not use this property in the proof of \textbf{(2)}$\Rightarrow$\textbf{(3)}.

From $\bm{(1_B)}$ we have
\begin{equation*}
	{b}^{0_{A}}={b}^{0_{A}+0_{A}}=\left({b}^{0_{A}}\right)^{0_{A}}+{b}^{0_{A}}
\end{equation*}
since $a^{0_{A}}=a$ for any $a\in A$, we obtain
\begin{equation*}
	{b}^{0_{A}}={b}^{0_{A}}+{b}^{0_{A}}
\end{equation*}
and then ${b}^{0_{A}}=0_A$.

Now we shall prove that the semi-direct product $B\ltimes A\in\Gwa$. Obviously, $B\ltimes A$ is a group as it is in the case of groups. We have to show the following equalities for any $(b,a), (b',a'), (b'',a'')\in B\ltimes A$:
\begin{enumerate}[label={\textbf{(\alph{*})}}, leftmargin=1cm]
	\item ${(b,a)}^{(b',a')+(b'',a'')}=\left({(b,a)}^{(b',a')}\right)^{(b'',a'')}$
	\item $\left({(b,a)}+{(b',a')}\right)^{(b'',a'')}={(b,a)}^{(b'',a'')}+{(b',a')}^{(b'',a'')}$
	\item ${(b,a)}^{(0_{B},0_{A})}={(b,a)}$.
\end{enumerate}

First we prove the equality in \textbf{(a)}.

\begin{enumerate}[leftmargin=1cm]
	\item[\textbf{(a)}] We have
	\begingroup\makeatletter\def\f@size{10}\check@mathfonts
	\def\maketag@@@#1{\hbox{\m@th\large\normalfont#1}}%
	\begin{alignat*}{2}
	{(b,a)}^{(b',a')+(b'',a'')} &= {(b,a)}^{(b'+b'',a'+b'\cdot a'')} \\
	&= \left( {b}^{b'+b''}, \left({a}^{a'+b'\cdot a''} \right)^{b'+b''}+\left( {b}^{a'+b'\cdot a''}\right)^{b'+b''}  \right)\\
	&= \left(\left({b}^{b'}\right)^{b''}, \left(\left(\left({a}^{a'}\right)^{b'\cdot a''}\right)^{b'}\right)^{b''}+\left(\left({b}^{a'}\right)^{b'\cdot a''}+{b}^{b'\cdot a''}\right)^{b'+b''}  \right)\\
	&= \left(\left({b}^{b'}\right)^{b''}, \left(\left(\left({a}^{a'}\right)^{b'}\right)^{a''}\right)^{b''}+\left(\left(\left({b}^{a'}\right)^{b'\cdot a''}\right)^{b'}\right)^{b''}+\left(\left({b}^{b'\cdot a''}\right)^{b'}\right)^{b''}\right)\\
	&= \left(\left({b}^{b'}\right)^{b''}, \left(\left(\left({a}^{a'}\right)^{b'}\right)^{a''}\right)^{b''}+\left(\left(\left({b}^{a'}\right)^{b'}\right)^{a''}\right)^{b''}+\left(\left({b}^{b'}\right)^{a''}\right)^{b''}\right).
	\end{alignat*} \endgroup
	On the other hand
	\begingroup\makeatletter\def\f@size{10}\check@mathfonts
	\def\maketag@@@#1{\hbox{\m@th\large\normalfont#1}}%
	\begin{alignat*}{2}
	\left({(b,a)}^{(b',a')}\right)^{(b'',a'')} &= \left({b}^{b'}, \left({a}^{a'}\right)^{b'}+\left({b}^{a'}\right)^{b'}\right)^{(b'',a'')}\\
	&= \left(\left({b}^{b'}\right)^{b''}, \left(\left(\left({a}^{a'}\right)^{b'}+\left({b}^{a'}\right)^{b'}\right)^{a''}\right)^{b''}+\left(\left({b}^{b'}\right)^{a''}\right)^{b''}\right)\\
	&= \left(\left({b}^{b'}\right)^{b''}, \left(\left(\left({a}^{a'}\right)^{b'}\right)^{a''}\right)^{b''}+\left(\left(\left({b}^{a'}\right)^{b'}\right)^{a''}\right)^{b''}+\left(\left({b}^{b'}\right)^{a''}\right)^{b''}\right)
	\end{alignat*} \endgroup
	From which we conclude that condition \textbf{(a)} holds in $B\ltimes A$.
\end{enumerate}
	Now we check condition \textbf{(b)}.
\begin{enumerate}[leftmargin=1cm]
	\item[\textbf{(b)}] We have
	\begingroup\makeatletter\def\f@size{10}\check@mathfonts
	\def\maketag@@@#1{\hbox{\m@th\large\normalfont#1}}%
	\begin{alignat*}{2}
	\left({(b,a)}+{(b',a')}\right)^{(b'',a'')} &= \left(b+b',a+b\cdot a'\right)^{(b'',a'')}\\
	&= \left(\left(b+b'\right)^{b''},\left(\left(a+b\cdot a'\right)^{a''}\right)^{b''}+\left(\left(b+b'\right)^{a''}\right)^{b''}\right) \\
	&= \left(b^{b''}+{b'}^{b''}, \left({a}^{a''}\right)^{b''}+\left(\left(b\cdot a'\right)^{a''}\right)^{b''}+\left({b}^{a''}\right)^{b''}+\left(b\cdot\left({b'}\right)^{a''}\right)^{b''}\right) \\
	&= \left(b^{b''}+{b'}^{b''}, \left({a}^{a''}\right)^{b''}+\left({b}^{a''}\right)^{b''}+\left(b\cdot\left({a'}^{a''}\right)\right)^{b''}+\left(b\cdot\left(b'\right)^{a''}\right)^{b''}\right)
	\end{alignat*}\endgroup
	Here we apply condition $\bm{(3_{A})}$. We have the following equalities.
	\begingroup\makeatletter\def\f@size{9.5}\check@mathfonts
	\def\maketag@@@#1{\hbox{\m@th\large\normalfont#1}}%
	\begin{alignat*}{2}
	{(b,a)}^{(b'',a'')}+{(b',a')}^{(b'',a'')} &= \left({b}^{b''},\left({a}^{a''}\right)^{b''}+\left({b}^{a''}\right)^{b''}\right)+\left({b'}^{b''},\left({a'}^{a''}\right)^{b''}+\left({b'}^{a''}\right)^{b''}\right)\\
	&= \left({b}^{b''}+{b'}^{b''}, \left({a}^{a''}\right)^{b''}+\left({b}^{a''}\right)^{b''} + {b}^{b''}\cdot \left({a'}^{a''}\right)^{b''} + {b}^{b''}\cdot \left({b'}^{a''}\right)^{b''}\right)
	\end{alignat*}\endgroup
	Applying condition $\bm{(4_{A})}$ we obtain
	\begin{equation*}
	b\cdot\left({a'}^{a''}\right)^{b''}={b}^{b''}\cdot \left({a'}^{a''}\right)^{b''}
	\end{equation*}
	and
	\begin{equation*}
	\left(b\cdot\left({b'}\right)^{a''}\right)^{b''}={b}^{b''}\cdot \left({b'}^{a''}\right)^{b''}
	\end{equation*}
	which proves that we have condition \textbf{(b)}.
\end{enumerate}
Finally, we check condition \textbf{(c)}.
\begin{enumerate}[leftmargin=1cm]
	\item[\textbf{(c)}] Here, if we apply the equalities
	\begin{equation*}
	\left({a}^{0_{A}}\right)^{0_{B}}={a}^{0_{B}}=a
	\end{equation*}
	and
	\begin{equation*}
	\left({b}^{0_{A}}\right)^{0_{B}}={0_{A}}^{0_{B}}=0_A,
\end{equation*}
	then we get
	\begin{equation*}
	 {(b,a)}^{(0_{B},0_{A})}=\left({b}^{0_{B}},\left({a}^{0_{A}}\right)^{0_{B}}+\left({b}^{0_{A}}\right)^{0_{B}}\right)={(b,a)}.
	\end{equation*}
\end{enumerate}
\item[\textbf{(3)}$\Rightarrow$\textbf{(1):}] Suppose $B\ltimes A\in\Gwa$, then we have a split extension
\begin{equation}\label{sext}
\xymatrix{
0 \ar[r] &   A \ar@{->}[r]^-{i} &  B\ltimes A \ar@{->}[r]_-{p} &   B \ar@/_/[l]_-{j} \ar[r] & 0 }
\end{equation}
where $p(b,a)=b$, $i(a)=(0,a)$ and $j(b)=(b,0)$. Define derived actions from this extension in a usual way.
\begin{alignat*}{2}
	b~\overline{\cdot}~a &= j(b)+a-j(b) \\
	&= (b,0)+(0,a)-(b,0) \\
	&= (b,b\cdot a)+(-b,0) \\
	&= (0,b\cdot a),
\end{alignat*}
therefore the derived action corresponding to the addition operation coincides with the given action.

Action corresponding to the action operation, denoted by $\ast$, is defined by
\begin{alignat*}{2}
	a*b &= \left(0_B,a\right)^{\left(b,0_A\right)} \\
	&= \left({0_{B}}^{b},\left({a}^{0_{A}}\right)^{b}+\left({0_{B}}^{0_{A}}\right)^{b}\right) \\
	&= \left(0,{a}^{b}\right).
\end{alignat*}
As we see this action also coincides with the given action.

For the dual to $\ast$ operation, i.e. dual action we have
\begin{alignat*}{2}
a\asc b &= \left(b,0_A\right)^{\left(0_B,a\right)}-\left(b,0_A\right) \\
&= \left(b^{0_{B}},\left({0_{A}}^{a}\right)^{0_{B}}+\left({b}^{a}\right)^{0_{B}}\right)-\left(b,0_A\right) \\
&= \left(b,b^{a}\right)-\left(b,0_A\right) \\
&= \left(b-b,b^{a}+b\cdot 0_A\right) \\
&= \left(0_B,b^{a}\right).
\end{alignat*}
Therefore this action also coincides with the given action of $B$ on $A$, which proves that the given action of $B$ on $A$ is a derived action, which concludes the proof of the theorem.
\end{enumerate}	
\end{proof}

For the examples of derived actions in the category $\Gwa$ see Section \ref{sect:subcat}, Lemma \ref{lem:selfact} and Corollary \ref{cor:idealact}.

\section{The subcategory $\rGwa\hookrightarrow\Gwa$}\label{sect:subcat}
Consider the objects $A\in\Gwa$ which satisfy two conditions:
\begin{enumerate}[label={\textbf{(\arabic{*})}}, leftmargin=1.7cm]
	\item ${x}^{y}+z=z+{x}^{y}$, $y\neq 0$ and
	\item ${x}^{\left({y}^{z}\right)}={x}^{y}$,
\end{enumerate}
for any $x,y,z\in A.$ This kind of objects will be called \textit{reduced groups with action}, and the corresponding full subcategory of $\Gwa$ will be denoted by $\rGwa.$

Derived actions are defined in $\rGwa$ in analogous way as it is in $\Gwa$.

\begin{example}
For any set $X$ let $F(X)$ be a free group with action on itself in $\Gwa$; one can see the construction in \cite{Datuashvili2004}. Let $R$ be a congruence relation on $F(X)$ generated by the relations
\begin{equation*}
{x}^{y}+z \sim z+{x}^{y}
\end{equation*}
for any $y\neq 0$ and
\begin{equation*}
{x}^{\left({y}^{z}\right)}\sim {x}^{y}
\end{equation*}
for any $x,y,z\in F(X)$. Then the quotient object 
$\faktor{F(X)}{R}$ by the $R$ obviously is an object of $\rGwa$.
\end{example}

\begin{theorem}\label{theo:equivcond}
Let $A,B\in\rGwa$ and $\beta=(\beta_{+}, \beta_{\ast}, \beta_{\asc})\colon B\times A\rightarrow A$ be a triple of actions of $B$ on $A$ in $\rGwa$. Then the following conditions are equivalent:
\begin{enumerate}[label={\textbf{(\arabic{*})}}, leftmargin=1cm]
\item $\beta$ is a triple of derived actions in $\rGwa$.
\item $\beta$ satisfies condition \textbf{(2)} of Theorem \ref{theo:der-act-equiv} and the following conditions
\begin{equation}\label{tennew}
\begin{array}{rclcrcl}
	b\cdot{a}^{a'}              & = & {a}^{a'} & \qquad\qquad\qquad\qquad & {a}^{b}+a'                  & = & a'+ {a}^{b} \\
	b\cdot{a}^{b'}              & = & {a}^{b'} & \qquad\qquad\qquad\qquad & {a}^{\left({a'}^{b}\right)} & = & {a}^{a'}    \\
	{b}^{b'}\cdot a             & = & a        & \qquad\qquad\qquad\qquad & {a}^{\left({b}^{a'}\right)} & = & a           \\
	{b}^{\left({a}^{a'}\right)} & = & {b}^{a}  & \qquad\qquad\qquad\qquad & {b}^{\left({b'}^{a}\right)} & = & 0      \\
	{a}^{\left({b}^{b'}\right)} & = & {a}^{b}  & \qquad\qquad\qquad\qquad & {b}^{\left({a}^{b'}\right)} & = & {b}^{a}
\end{array}	
\end{equation}
for any $a,a'\in A$, $b,b'\in B$. Note that under the conditions \eqref{tennew}, $\bm{(2_{A})}$, $\bm{(3_{A})}$ and $\bm{(4_{A})}$ have simpler forms.
\item The semi-direct product $B\ltimes A$ is an object in $\rGwa$.
\end{enumerate}
\end{theorem}

\begin{proof} \begin{enumerate}[leftmargin=2.2cm]
\item[\textbf{(1)}$\Rightarrow$\textbf{(2):}] We will check only  the conditions ${a}^{\left({b}^{a'}\right)}=a$, ${b}^{\left({b'}^{a}\right)}=0$ and ${b}^{\left({a}^{b'}\right)}={b}^{a}$. Other conditions are obvious.
	\begin{enumerate}[label=\textbf{(\roman{*})},leftmargin=1cm]
	\item ${a}^{\left({b}^{a'}\right)}={a}^{\left({j(b)}^{a'}\right)-j(b)}={a}^{{j(b)}-j(b)}={a}^{0}=a$;
	\item ${b}^{\left({b'}^{a}\right)}={j(b)}^{\left({j(b')}^{a}-j(b')\right)}-j(b)=\left({j(b)}^{j(b')}\right)^{-j(b')}-j(b)={j(b)}^{0}-j(b)=0$;
	\item ${b}^{\left({a}^{b'}\right)}={j(b)}^{\left({a}^{j(b')}\right)}-j(b)={j(b)}^{a}-j(b)={b}^{a}$.
	\end{enumerate}
\item[\textbf{(2)}$\Rightarrow$\textbf{(3):}] By Theorem \ref{theo:der-act-equiv} we need to prove only that
\begin{equation*}
	{(b,a)}^{(b',a')}+(b'',a'')=(b'',a'')+{(b,a)}^{(b',a')}
\end{equation*}
and
\begin{equation*}
	{(b,a)}^{\left({(b',a')}^{(b'',a'')}\right)}={(b,a)}^{(b',a')}
\end{equation*}
for any $(b,a),(b',a'),(b'',a'')\in B\ltimes A$. We have
\begin{alignat*}{2}
{(b,a)}^{(b',a')}+(b'',a'') &= \left({b}^{b'},{\left({a}^{a'}\right)}^{b'}+{\left({b}^{a'}\right)}^{b'}\right)+(b'',a'') \\
&= \left( {b}^{b'}+b'',{\left({a}^{a'}\right)}^{b'}+{\left({b}^{a'}\right)}^{b'}+{b}^{b'}\cdot a''\right) \\
&= \left( {b}^{b'}+b'',{\left({a}^{a'}\right)}^{b'}+{\left({b}^{a'}\right)}^{b'}+ a''\right).
\end{alignat*}
On the other hand
\begin{alignat*}{2}
(b'',a'')+{(b,a)}^{(b',a')} &= (b'',a'')+\left({b}^{b'},{\left({a}^{a'}\right)}^{b'}+{\left({b}^{a'}\right)}^{b'}\right) \\
	&= \left(b''+{b}^{b'},a''+b''\cdot{\left({a}^{a'}\right)}^{b'}+b''\cdot{\left({b}^{a'}\right)}^{b'}\right) \\
	&= \left({b}^{b'}+b'',a''+{\left({a}^{a'}\right)}^{b'}+{\left({b}^{a'}\right)}^{b'}\right) \\
	&= \left({b}^{b'}+b'',{\left({a}^{a'}\right)}^{b'}+{\left({b}^{a'}\right)}^{b'}+a''\right).
\end{alignat*}
which proves the first identity. For the second identity we have
\begingroup\makeatletter\def\f@size{11}\check@mathfonts
\def\maketag@@@#1{\hbox{\m@th\large\normalfont#1}}%
\begin{alignat*}{2}
	{(b,a)}^{\left({(b',a')}^{(b'',a'')}\right)} & = {(b,a)}^{\left({b'}^{b''},\left({a'}^{a''}\right)^{b''}+\left({b'}^{a''}\right)^{b''}\right)} \\
	&= \left({b}^{\left({b'}^{b''}\right)}, {\left({a}^{\left({\left({a'}^{a''}\right)}^{b''}+{\left({b'}^{a''}\right)}^{b''}\right)}\right)}^{\left({b'}^{b''}\right)}+{\left({b}^{\left({\left({a'}^{a''}\right)}^{b''}+{\left({b'}^{a''}\right)}^{b''}\right)}\right)}^{\left({b'}^{b''}\right)}\right) \\
	&= \left({b}^{b'}, {\left({\left({a}^{\left({\left({a'}^{a''}\right)}^{b''}\right)}\right)^{\left({\left({b'}^{a''}\right)}^{b''}\right)}}\right)}^{b'}+{\left({b}^{\left({\left({a'}^{a''}\right)}^{b''}+{\left({b'}^{a''}\right)}^{b''}\right)}\right)}^{b'}\right) \\
	&= \left({b}^{b'}, \left({a}^{a'}\right)^{b'}+{\left({\left({b}^{\left({a'}^{a''}\right)}\right)}^{\left({\left({b'}^{a''}\right)}^{b''}\right)}+{b}^{\left({\left({b'}^{a''}\right)}^{b''}\right)}\right)}^{b'}\right) \\
	&= \left({b}^{b'}, \left({a}^{a'}\right)^{b'}+{\left({\left({b}^{a'}\right)}^{\left({b'}^{a''}\right)}\right)}^{b'}+{\left({b}^{\left({b'}^{a''}\right)}\right)}^{b'}\right) \\
	&= \left({b}^{b'}, \left({a}^{a'}\right)^{b'}+{\left({b}^{a'}\right)}^{b'}\right) \\
	&= {(b,a)}^{(b',a')}
\end{alignat*}\endgroup
which proves the second identity. Here we applied that ${\left({b}^{\left({b'}^{a''}\right)}\right)}^{b'}=0$, which follows from \eqref{tennew}, where we have $b^{\left({b'}^{a}\right)}=0$, for any $a\in A,$ in particular for $a=a''$ in our case, and the fact that $0^{b'}=0$ (\ref{prop:actcond} (b)).
\item[\textbf{(3)}$\Rightarrow$\textbf{(1):}] The proof is the same as of the one in Theorem \ref{theo:der-act-equiv} and therefore we omit.
\end{enumerate}	
\end{proof}

\begin{lemma}\label{lem:selfact}
Let $A\in \Gwa$ (resp. $A\in \rGwa$). An action of $A$ on itself defined by $a\cdot a'=a+a'-a,  a'\ast a= {a'}^{\triangleright a}={a'}^{a}$ and $a'\asc a=a^{\circ a'}={a}^{a'}-a$, for $a,a'\in A$, is a derived action in $\Gwa$ (resp. $\rGwa$).
\end{lemma}
\begin{proof} Easy but careful checking of the conditions given in Theorem \ref{theo:der-act-equiv} (resp. Theorem \ref{theo:equivcond}).
\end{proof}

Note, that an action of $A$ on itself defined by $a\cdot a'=a+a'-a, a'^{\triangleright a}=a'^a$ and $a^{\circ a'}=a^{a'}$, for $a,a'\in A,$ is not a derived action in $\Gwa$ and therefore in $\rGwa.$ It is obvious that conditions $\bm{(2_A)}$ and $\bm{(1_B)}$ are not satisfied.

\begin{corollary}\label{cor:idealact}
Let $A\in \Gwa$ (resp. $A\in \rGwa$) and let $I\subset A$ be an ideal of $A.$ Then the action of $A$ on $I$ defined by $a\cdot i=a+i-a, i^{\triangleright a}=i^a$ and $a^{\circ i}=a^i-a$, $i\in I$, $a\in A$ is a derived action in $A\in\Gwa$ (resp. in $A\in\rGwa$).
\end{corollary}

Lemma \ref{lem:selfact} and Corollary \ref{cor:idealact} give examples of derived actions in the categories $\Gwa$ and $\rGwa$.


\end{document}